\documentclass[a4paper]{article}

\usepackage[utf8]{inputenc}
\usepackage[T1]{fontenc}
\usepackage[francais,english]{babel}
\usepackage{refcount}
\usepackage{amsmath,amssymb,amsfonts,amsthm}
\usepackage{mathrsfs}
\usepackage{amsfonts}
\usepackage{amssymb}
\usepackage{latexsym}
\usepackage{comment}
\usepackage[normalem]{ulem}
\usepackage[colorlinks=true, linkcolor=red, citecolor=red, urlcolor=blue]{hyperref}

\usepackage{xcolor}

\newcommand{\RG}{\color{blue}}

\def\R{\mathbb R}
\def\N{\mathbb N}

\def\C{\mathbb C}

\def\O{\Omega}
\def\a{\alpha}

\def\et0{e^{tA}x_0}

\numberwithin{equation}{subsection}

\makeatletter
\renewcommand{\theequation}{%
  \ifnum\value{subsection}=0
    \mbox{\thesection.\arabic{equation}}%
  \else
    \mbox{\thesubsection.\arabic{equation}}%
  \fi
}
\makeatother

\newtheorem{thm}{Theorem}

\newtheorem{Def}{Definition}

\newtheorem{Cor}{Corollary}
\newtheorem{rk}{Remark}
\newenvironment{contrk}[1]{%
 
  \setcounter{rk}{\getrefnumber{#1}}%

  \addtocounter{rk}{-1}%

  \begin{rk}[\textbf{continued}]%
}{%
  \end{rk}%
}

\author{
\textsc{Shri Lal Raghudev Ram Singh,}
\footnote{
University of Waterloo, email: \texttt{slrrsingh@uwaterloo.ca}
}
P. \textsc{Cannarsa,}
\footnote{Universit\`a di Roma Tor Vergata, via della Ricerca Scientifica 1, 00133, Roma, Italy, 
email: \texttt{cannarsa@axp.mat.uniroma2.it} (corresponding author)} 
R. \textsc{Guglielmi}
\footnote{University of Waterloo,
email: \texttt{roberto.guglielmi@uwaterloo.ca}}
\thanks{This research has been performed in the framework of the GDRE CONEDP.
}
}

\title{Uniform stabilization for relatively bounded perturbations of generators of semigroup}
\date{}

\begin{document}

\maketitle

\begin{abstract}
In this paper, we study the robustness of exponential stability for semigroups generated by linear operators under perturbations. Extending a classical result of Gibson’s Stability Theorem, we show that if the generator of an analytic exponentially stable semigroup is perturbed by a class of relatively bounded operators satisfying certain assumptions, then exponential stability is preserved, provided the perturbed semigroup is strongly stable. We also show that, for a restricted class of perturbations, the analyticity requirement can be relaxed to Gevrey regularity. Moreover, we present applications to uniformly parabolic equations, degenerate/singular parabolic equations, coupled hyperbolic plate systems, and generalized coupled systems of Kirchhoff-Love plates and a membrane-like electric network.

\end{abstract}

\bigskip
\noindent
\textbf{Key words: } Analytic Semigroup, Gevrey's Semigroup, robust exponential stability, stabilization, linear evolution equations.

\smallskip
\noindent
\textbf{AMS subject classifications: } Primary: 47A55, 47D06; Secondary: 93D23, 34G10, 35L10, 93D20.

%%%%%%%%%%%%%%%%%%%%%%%%%%%%%%%%%%%%%%%%%%%%%%%%%%%%%%%%%%%%%%%%%%%%%%%%%%%%%%%%%%
%%%%%%%%%%%%%%%%%%%%%%%%%%%%%%%%%%%%%%%%%%%%%%%%%%%%%%%%%%%%%%%%%%%%%%%%%%%%%%%%%%
\section{Introduction}\label{intro}
%%%%%%%%%%%%%%%%%%%%%%%%%%%%%%%%%%%%%%%%%%%%%%%%%%%%%%%%%%%%%%%%%%%%%%%%%%%%%%%%%%
%%%%%%%%%%%%%%%%%%%%%%%%%%%%%%%%%%%%%%%%%%%%%%%%%%%%%%%%%%%%%%%%%%%%%%%%%%%%%%%%%%

%{\RG \large Double check capital letters in the reference list.\\ }
%---------------------------------------------------------------------------------
A classical problem in control theory concerns the stabilization of linear and nonlinear Partial Differential Equations (PDEs). Different approaches to this problem have been developed; see, for example, control-theoretic method \cite{Russell-1975}, stabilization using linear feedback operators \cite{Gibson-1980, Russell-1969, Russell-1972, Slemrod-1974, Triggiani-1989}, and indirect stabilization \cite{Russell-1993, Kapitonov-1996}.

A well-known result by Gibson \cite{Gibson-1980} ensures the permanence of exponential stability for a proper perturbation applied to the generator of a uniformly stable semigroup of bounded operators. Specifically, Gibson's stability theorem guarantees that exponential stability is preserved under compact perturbations of the generator of an exponentially stable semigroup, provided the perturbed semigroup is strongly stable. This result was generalized in \cite{AB-C}, where the authors relaxed the requirement of compactness of the perturbation operator, replacing it with an assumption on the boundedness of the perturbation operator combined with its compactness on trajectories of the system. This simple observation significantly broadens the class of systems that satisfy this condition, and it is based on a constructive proof of Gibson's result on reflexive Banach spaces.% for bounded perturbations.

The result in~\cite{AB-C} shows how one can relax a regularity assumption on the perturbation operator and replace it with a suitable assumption on the interplay between the perturbation and the dynamics itself. Their constructive proof hints to the possibility of appropriately tuning the regularity of the perturbation operator in terms of the regularity of the system dynamics in order to preserve the exponential stability of the perturbed system. In this paper, we investigate in depth such interplay and extend Gibson's stability result to accommodate relatively bounded perturbations. % (see Theorem~\ref{thmanalyticgib} and, while keeping the  exponential stability robust. 
The main result of the paper, Theorem~\ref{thmanalyticgib} in Section~\ref{sec:main}, states that if the generator of an analytic exponentially stable semigroup is perturbed by a relatively bounded perturbation which is compact on trajectories of the system, then the perturbed operator generates an exponentially stable semigroup, provided it is strongly stable.
\begin{comment}
     Indeed, it is well known \cite{CurZwa} that the standard heat equation
\begin{equation}
\left\{
\begin{array}{ll}
y_t(x,t) = y_{xx}(x,t) + \mu y(x,t) & \text{on } (0,L)\times (0,+\infty) \\
y(0,t) = y(L,t) = 0 & \text{on } (0,+\infty) \\
y(x,0) = y_0(x) & \text{on } (0,L)
\end{array}
\right.
\end{equation}
is unstable for values $\mu\ge \lambda_1$, where $\lambda_1$ stands for the least eigenvalue of the operator $-\partial^2_{xx}$ on $H^1_0(0,L)$.
\end{comment}
%{\RG [Complete Introduction including Alabau-Cannarsa, Gibson, Triggiani, and so on...]}
The proof is constructive, and follows similar arguments as in~\cite{AB-C}. This allows us to generalize the main result to generators of exponentially stable semigroups which are not analytic.  Indeed, we can relax the regularity assumption of analyticity by replacing it with Gevrey regularity of a given order, while considering a smaller class of relatively bounded perturbations of appropriate regularity (see Corollary~\ref{gev-theta-sub}).

%Towards the end, we also provide two applications of our result to the case of perturbations applied to the generator of an exponentially stable semigroup, which is not analytic but belongs to some Gevrey class.

The rest of the paper is organized as follows. In the next section, we introduce and compare different notions of relatively bounded perturbations, and we state and prove our main results. In Section~\ref{sec:applic}, we illustrate the applications of our results to prove the exponential stability of several classes of partial differential equations, including a uniformly parabolic equation in Section~\ref{uniparsec}, a degenerate/singular parabolic equations in Section~\ref{degsingsec},  coupled plate dynamics in Section~\ref{coupledplate}, and lastly, the generalized coupled system of Kirchhoff-Love plates and a membrane-like electric network in Section~\ref{couplekirchhof-elect}.

\section{Main result}\label{sec:main}

We begin this section by defining three notions of relatively bounded perturbations of a linear operator $A:D(A)\to X$, where $X$ is a reflexive Banach space with norm $\|\cdot\|$. We will denote the standard linear operator norm by $\|\cdot\|_{\mathcal{L}}$.

\begin{Def}[$A$-boundedness, \cite{kato}]\label{Abound}
A linear operator $B:D(B)\to X$ is said to be \emph{$A$-bounded} if \( D(A) \subseteq D(B) \), and there exist constants $a, b \geq 0$  such that 

\begin{equation}
\|Bx\| \;\le\; a\,\|x\| \;+\; b\,\|Ax\|,\qquad \forall x\in D(A).
\end{equation}

The greatest lower bound $b_0$ of all constants $b$ for which \eqref{Abound} holds is called the \emph{$A$-bound} of $B$ with respect to $A$.
\end{Def}

\begin{Def}[$\theta$-subordination {\cite[Chapter 1]{Mar88}}]\label{theta-sub}
A linear operator $ B : D(B) \to X $ is said to be \emph{$\theta$-subordinate} to $A$ if $D(A) \subseteq D(B)$  and, for some \( 0 \leq \theta < 1 \), there exists a constant $c>0$ such that
\[
\| Bx \| \le c\, \| Ax \|^{\theta} \| x \|^{1 - \theta}, \qquad \forall x \in D(A).
\]
\(B\) is also termed \emph{subordinate of order} \(\theta\) in the literature (see, for example, \cite[Definition 7.2]{Krein1971}).

\end{Def}
\begin{Def}[$(-A)^\theta$-boundedness]\label{athetabound}
A linear operator $B:D(B)\to X$ is called \emph{$(-A)^\theta$-bounded} if $D((-A)^\theta)\subseteq D(B)$ for some $0\leq \theta <1$, and there exists a constant $d>0$ such that 
\[
\|Bx\|\leq d\,\|(-A)^\theta x\|, \quad \forall x\in D((-A)^\theta).
\]

\end{Def}
Note that, if $A$ is the infinitesimal generator of an analytic semigroup, then $(-A)^\theta$-boundedness implies $\theta$-subordination with the same $0\le\theta < 1$ (although the converse is not true \cite{Markus-Mat1982}), and $\theta$-subordination for any $0\le\theta < 1$ implies $A$-boundedness. Indeed, by the classical interpolation inequality \cite[Theorem~6.10]{Pazy} and Young's inequality, for any $x \in D(A)$ and $0\le\theta < 1$, there exists a constant $c_\theta>0$ such that %{\RG [shall we denote such constant $c$ with a different symbol than $c$ in Definition~\ref{theta-sub}? For example, $c_\theta$.]}
\begin{equation}\label{rel-theta-Abound-relation}
 \|(-A)^\theta x\| 
   \leq c_\theta \|Ax\|^\theta \|x\|^{1 - \theta}
   \leq c_\theta(1-\theta)\|x\| + c_\theta\,\theta\|Ax\|.
\end{equation}
%{\RG [... and removing coefficient $d$ in all subsequent terms of the inequalities.]}
Thus, $A$-boundedness is the most general condition  among these three notions of relatively boundedness, whereas $(-A)^\theta$-boundedness is the most restrictive. We now present the main result of the paper, which ensures the robustness of exponential stability under $\theta$-subordinate perturbations.
\begin{thm}\label{thmanalyticgib}
Let $(A,D(A))$ be the generator of an analytic semigroup  $\left(e^{tA}\right)_{t\ge 0}$ on a reflexive Banach space $X$. Consider a linear operator $B:D(B)\to X$ such that 
\begin{itemize}
\item[i)] B is $\theta$-subordinate to A in the sense of Definition \ref{theta-sub};
\item[ii)] there exist $\omega > 0$ and $M>0$ such that $\|e^{tA}\|_{\mathcal{L}}\le M e^{-\omega t}$ for all $t\ge 0$;
\item[iii)] $Be^{tA}$ is compact for all $t>0$;
\item[iv)] $e^{t(A+B)}x\to 0$ as $t\to +\infty$ for all $x\in X$.
\end{itemize}
Then there exist positive constants $\tilde{M}$ and $\tilde{\omega}$ such that $\|e^{t(A+B)}\|_{\mathcal{L}} \le \tilde{M}e^{-\tilde{\omega}t}$ for all $t\ge 0$.
\end{thm}
 The proof follows closely \cite[Theorem 1.1]{AB-C}, which assumes that $(A,D(A))$ generates a strongly continuous semigroup on $X$ and $B$ is a bounded operator. Here instead we relax the boundedness assumption on $B$, while requiring more regularity on the semigroup $\left(e^{tA}\right)_{t\ge 0}$. %In Corollary~\ref{thmthetabound} we will generalize Theorem~\ref{thmanalyticgib} by demonstrating the interplay between the regularity of the semigroup $e^{tA}$ and the operator $B$ to preserve the exponential stability of the perturbed semigroup.
%While in \cite{AB-C} the authors relied on the boundedness of the perturbation $B$, we here exploit the analyticity of $A$ to obtain estimate for $A e^{tA}$, since for $\theta$-subordinate perturbations, $B$ can be controlled in terms of $A$. We include the proof for completeness.

 %{\RG [comment also on the portion of the proof that requires care since it differs from the proof of~\cite[Theorem 1.1]{AB-C}.]} %We include it for completeness.
\begin{proof}

Notice that the Banach-Steinhaus Theorem ensures that
\begin{equation}\label{ubp}
\|e^{t(A+B)}\|_{\mathcal{L}} \le M_1\quad \forall t\ge 0\, ,
\end{equation}
for some constant $M_1 > 0$. 
From a well-known characterization of exponential stability for strongly continuous semigroups \cite{EnNa}, our claim is equivalent to show the limit
\begin{equation}\label{stpun}
\lim_{t\to +\infty} \|e^{t(A+B)}\|_{\mathcal{L}} = 0\, .
\end{equation}
Let us now define
\begin{equation}
\Lambda_t x := \int_0^t e^{(t-s)(A+B)}Be^{sA}x\;   \mathrm{d}s\, ,\quad\forall x\in X\, ,\ \forall t\ge 0\, .
\end{equation}
%{\RG [Just a comment on latex aesthetics: let's write the differentials in the integrals as follows:
%\[
%\Lambda_t x := \int_0^t e^{(t-s)(A+B)}Be^{sA}x\; %\mathrm{d}s\, ,
%\]
%which is more pleasant (and clear). Most integrals appear in the Applications section. Another layout comment: here we use $\forall x\in X$ after the equation. We shall make this notation uniform also with the previous part of the section with the definitions.]
Since $y(t) = e^{t A} x$ satisfies
\begin{equation*}
\begin{cases}
y'(t) = (A+B)y(t) - B y(t) \\
y(0) = x\, ,
\end{cases}
\end{equation*}
from the variation of constants formula we deduce that
\begin{equation}\label{varofcons}
e^{t(A+B)}x = e^{tA}x + \Lambda_t x,\, \quad\forall x\in X\, ,\ \forall t\ge 0\, .
\end{equation}
Thus, we have
\begin{equation}\label{stpoi}
\|e^{t(A+B)}\|_{\mathcal{L}} 
\le \sup_{|x|\le 1}\left\|e^{tA}x\right\| + \sup_{|x|\le 1}\left\|\Lambda_t x\right\|
\le M e^{-\omega t} + \sup_{|x|\le 1}\left\|\Lambda_t x\right\|\, .
\end{equation}
Consider now any sequence $t_n$ of positive numbers diverging as $n\to +\infty$, and choose vectors $x_n\in X$ such that $|x_n|\le 1$ and
\begin{equation}\label{approx1}
\sup_{|x|\le 1}\left\|\Lambda_{t_n} x\right\| < \|\Lambda_{t_n}x_n\| + \frac{1}{n}\, .
\end{equation} 
Owing to the Banach-Alaoglu Theorem, there exists a subsequence (still denoted by $x_n$ for simplicity) weakly convergent to some limit $\bar{x}\in X$. {Observe that, combining the exponential stability of $A$, the strong stability of $A+B$ and~\eqref{varofcons}, we have 
\begin{equation}\label{eqLambbarxgoestozero}
 \lim_{n\rightarrow +\infty}\|\Lambda_{t_n}\bar{x}\| = \lim_{n\rightarrow +\infty} \|e^{t_n(A+B)}\bar{x}-e^{t_nA}\bar{x}\| =0. 
\end{equation}}
Next, we prove that
\begin{equation}\label{stconcaplmb}
\Lambda_{t_n}(x_n - \bar{x}) \to 0\quad \text{strongly as }\ n\to +\infty\, .
\end{equation}
Indeed, introduce the sequence of vector-valued functions
$$
\phi_n(s) := B e^{sA}(x_n - \bar{x})\, ,\quad \forall s\ge 0\, .
$$
Hypothesis $iii)$ ensures that $\phi_n(s)$ strongly converges to zero as $n\to +\infty$ for all $s>0$. Moreover, since $A$ is analytic, there exist positive constants $\omega$ and $M_2$ such that
$$\|Ae^{tA}\|_{\mathcal{L}}<\frac{M_2}{t}e^{-\omega t},\qquad \forall t>0,$$
 and thanks to hypothesis $i)$ and $ii)$ we have that
% \begin{align}\label{estfrctpw}
% \|\phi_n(s)\| &= \|B e^{sA}(x_n - \bar{x})\|\nonumber \\
% &\leq c\, \|A e^{sA}(x_n - \bar{x})\|^{\theta}\, \|e^{sA}(x_n - \bar{x})\|^{1 - \theta}\nonumber \\
% &\leq 2c\, \|A e^{sA}\|^{\theta} \, \|e^{sA}\|^{1 - \theta} \nonumber\\
% &\leq  2c\,M_2^\theta\, M^{1-\theta}\, \frac{e^{-\omega s}}{s^{\theta}},\qquad \forall s>0.
% \end{align}
\begin{multline}\label{estfrctpw}
\|\phi_n(s)\| = \|B e^{sA}(x_n - \bar{x})\| \leq c\, \|A e^{sA}(x_n - \bar{x})\|^{\theta}\, \|e^{sA}(x_n - \bar{x})\|^{1 - \theta} \\
\leq 2c\, \|A e^{sA}\|^{\theta} \, \|e^{sA}\|^{1 - \theta} 
\leq  2c\,M_2^\theta\, M^{1-\theta}\, \frac{e^{-\omega s}}{s^{\theta}},\qquad \forall s>0.
\end{multline}

%{\RG [In Definition 2, we called the constant $c$ rather than $b$. Just in case we want to keep a uniform notation with that too, just to avoid confusion.Before the equation, might be useful to recall the specific inequality that you are using here from the analyticity of $A$.\\Finally, should the exponent $\omega$ in $e^{-\omega s}$ depend on $\theta$, or does that cancel out? This should be clear after specifying the inequality for $A$ analytic.]\\} 
Therefore, Lebesgue's dominated convergence theorem for vector-valued functions implies that $\phi_n\to 0$ in $L^1(0,+\infty;X)$ as $n\to +\infty$. Hence, using inequality~\eqref{ubp} we get that
$$
\|\Lambda_{t_n}(x_n - \bar{x})\| \le M_1\int_0^\infty \|\phi_n(s)\| \;\mathrm{d}s \to 0\quad \text{as}\quad n\to +\infty\, ,
$$
which implies \eqref{stconcaplmb}. By collecting \eqref{stpoi}, \eqref{approx1}, \eqref{eqLambbarxgoestozero} and \eqref{stconcaplmb}, we conclude that
\begin{equation*}\label{finstp}
\|e^{t_n(A+B)}\|_{\mathcal{L}}\le Me^{-\omega t_n} + \frac{1}{n} + \|\Lambda_{t_n}(x_n - \bar{x})\| + \|\Lambda_{t_n}\bar{x}
\| \to 0\quad \text{as}\quad n\to +\infty\, .
\end{equation*}
%{\RG [We shall comment here why does $|\Lambda_{t_n}\bar{x}|\to 0$ as $n\to +\infty$. Does it go to zero? Combining \eqref{stpoi}, \eqref{approx1} and \eqref{stconcaplmb}, I get that the first three terms of the RHS go to zero.]\\}
Since $t_n$ is an arbitrary sequence tending to infinity, this implies \eqref{stpun}, which proves the desired claim.
\end{proof}

 %{\RG [Now that I read it, Corollary 1 is so obvious that does not seem worth stating it as a Corollary, nor proving it. We can discuss this further if you think differently. The only difference is in the assumption \emph{(i)}, right? Everything else in the statement seems just the same. See also comments after Remark 1 before making any changes]}.

%{\RG [The fact that the argument does not hold for $A$-bounded perturbations would be better highlighted in a remark environment.]}

A closer look at the constructive proof of Theorem \ref{thmanalyticgib} shows that the previous argument can be generalized to establish exponential stability under $\theta$-subordinate perturbations by weakening the analyticity assumption on $\left(e^{tA}\right)_{t\ge 0}$ to a suitable class of differentiable semigroups, while compensating for such loss of regularity by allowing a smaller class of $\theta$-subordinate perturbations $B$. 
\begin{Cor}\label{gev-theta-sub}
Let \( (A,D(A)) \) be the generator of a  differentiable semigroup $\left(e^{tA}\right)_{t\ge 0}$ on a reflexive Banach space $X$ such that
\begin{equation}\label{derivativeestimate}
    \|Ae^{tA}\|_{\mathcal{L}}\le Ct^{-\beta} \quad \text{for some} \quad C>0 \quad\text{and}\quad \beta \ge 1.
\end{equation}
Suppose a linear operator \( B:D(B)\to X \) satisfies:
\begin{itemize}
\item[(i)] B is $\theta$-subordinate to A for some \( \theta \in [0,1/\beta) \);
\item[(ii)] There exist constants \( M, \omega > 0 \) such that \( \|e^{tA}\|_{\mathcal{L}} \le M e^{-\omega t} \) for all \( t \ge 0 \);
\item[(iii)] The operator \( B e^{tA} \) is compact for all \( t > 0 \);
\item[(iv)] The perturbed semigroup satisfies \( e^{t(A+B)}x \to 0 \) as \( t \to +\infty \) for all \( x \in X \).
\end{itemize}
Then there exist positive constants $\tilde{M}$ and $\tilde{\omega}$ such that $\|e^{t(A+B)}\|_{\mathcal{L}} \le \tilde{M}e^{-\tilde{\omega}t}$ for all $t\ge 0$.
\end{Cor}
\begin{proof}
The proof follows the same lines as the proof of Theorem~\ref{thmanalyticgib} until equation~\eqref{eqLambbarxgoestozero}. However, to prove~\eqref{stconcaplmb}, we need to use a different argument in~\eqref{estfrctpw}. Thanks to the assumptions \emph{(i)}, \emph{(ii)} and~\eqref{derivativeestimate}, we get that
%the following bound on \( \|\phi_n(s)\| \)
\begin{multline}\label{estfrctpw2}
\|\phi_n(s)\| = \|B e^{sA}(x_n - \bar{x})\| \leq c\, \|A e^{sA}(x_n - \bar{x})\|^{\theta}\, \|e^{sA}(x_n - \bar{x})\|^{1 - \theta} \\
\leq 2c\, \|A e^{sA}\|^{\theta} \, \|e^{sA}\|^{1 - \theta} 
\leq  2c\,C^\theta\, M^{1-\theta}\, \frac{e^{-\omega s(1-\theta)}}{s^{\theta\beta}}\,,\qquad \forall s>0.
\end{multline}
% \begin{equation}\label{estfrctpw2}
% |\phi_n(s)| \leq 2c \|A e^{sA}\|^{\theta} \|e^{sA}\|^{1 - \theta}
% \leq \frac{2cC^\theta M^{1-\theta}}{s^{\theta\beta} }e^{-\omega s(1-\theta)}, 
% \quad \forall\, s>0 .
% \end{equation}
%where we used and exponential stability estimate, as well as the bound \( \|x_n - \bar{x}\| \le 2 \). For integrability, we require \
The condition \( \theta \beta < 1 \) ensures the integrability of the right-hand side as a function of $s$ on $(0,\infty)$, and therefore it implies the convergence of the vector valued functions \( \phi_n \) to zero in \( L^1(0,\infty; X) \) as \( n \rightarrow \infty \). The remainder of the proof then proceeds along the same lines as the proof of Theorem~\ref{thmanalyticgib}. 
\end{proof}
\begin{comment}
    \begin{rk}\label{gev-theta-sub}
    From the above proofs (see argument \eqref{estfrctpw} and \eqref{estfrctpw21}), it is clear that Theorem \ref{thmanalyticgib} and Corollary \ref{thmthetabound} will hold true under following conditions on A and B:
\begin{itemize}
\item[H1)] $(A,D(A))$ generates a differentiable semigroup $e^{tA}$ on a reflexive Banach space $X$ such that $\|Ae^{tA}\|\le Ct^{-\beta}$ for some $C>0$ and $\beta \ge 1$,
\item[H2)] \(B\) is \emph{$(-A)^\theta$-bounded} or \emph{$\theta$-subordinate} to \( A \) for some \( \theta \in [0,1/\beta) \),
\item[H3)] the same hypotheses i), ii), iii) as in Theorem \ref{thmanalyticgib}. {\RG [It should be the same hypotheses ii), iii), iv) as in the Theorem?]}
\end{itemize}
{\RG [Rather than stating and proving Corollary 1, it would probably make more sense to state Remark 1 as a Corollary (only with $\theta$-subordinate to \( A \)), and prove the slightly different estimate to get the conclusion.\\
I re-iterate here that would probably be more readable to state both Theorem 1 and Remark 1 only for $\theta$-subordinate to \( A \), and after both results have been stated to comment on the fact that both of them are still true if $B$ is $(-A)^\theta$-bounded.]\\
}
{
In particular, this means that we can weaken the analyticity assumption on $A$, while compensating such loss of regularity by allowing a smaller class of relatively bounded perturbations $B$ {\RG as described by H2)}.}
\end{rk}
\end{comment}

\begin{rk}
Corollary~\ref{gev-theta-sub} highlights the interplay between the regularity of the semigroup~$\left(e^{tA}\right)_{t\ge 0}$ in~\eqref{derivativeestimate} and of the operator $B$ in assumption {(i)} in order to preserve the exponential stability of the perturbed semigroup.
\end{rk}
\begin{rk}\label{rk:gevrey-sem}
It is worth noting that a differentiable semigroup satisfying~\eqref{derivativeestimate} for some $\beta\ge 1$ fulfills Gevrey regularity of class $\delta>\beta$~\cite{Tay}.  For the sake of completeness, we remind the definition of Gevrey semigroups.
\end{rk}

\begin{Def}
A strongly continuous semigroup $\left(T(t)\right)_{t\ge 0}$ on a Banach space $X$ is of Gevrey
class $\delta > 0$ for $t > t_0$ iff $\left(T(t)\right)_{t\ge 0}$ is differentiable for $t > t_0$ and for every compact $K \hookrightarrow (t_0,+\infty)$ and each $\mu > 0$
there exists a constant $C = C(\mu,K)$ such that
$$
\|T^{(n)}(t)\|_{\mathcal{L}}\le C\mu^n (n!)^\delta\quad \forall t\in K\, ,\ n = 0,1,2,\dots
$$
\end{Def}
\noindent
\begin{contrk}{rk:gevrey-sem}
 
Another sufficient condition for Gevrey regularity of class $\delta$ of the $C_0$-semigroup $\left(e^{tA}\right)_{t\ge 0}$ is given in terms of the following resolvent estimate \cite{Tay}: If
\begin{equation}\label{resolv-cond-gev}
\limsup_{|\lambda|\to\infty} |\lambda|^{1/\beta}\big\|(i\lambda I-A)^{-1}\big\|_{\mathcal{L}}<\infty \qquad \text{for some}\quad \beta\geq1,
\end{equation}
then $\left(e^{tA}\right)_{t\ge 0}$ is Gevrey of class $\delta > \beta$. This condition is often used to determine Gevrey regularity in the literature (see \cite{ Hao-Liu-Yong-2015, Grab-las, Lasiecka-Tebou}). %{\RG [This should be Sections 3.4 and 3.5? But at the moment we already say this at the end of the remark. Make sure to make this reference only once in the remark, wherever it sounds better.]}
This requires analyzing the connection between conditions~\eqref{derivativeestimate} and~\eqref{resolv-cond-gev}, since in general they are not equivalent. To this aim, notice that differentiable semigroups satisfying the estimate \eqref{derivativeestimate} belong to the Crandall-Pazy class on Banach spaces with parameter $1/\beta\in(0,1]$, which implies condition \eqref{resolv-cond-gev}~\cite{Crandall-Pazy}. However, \eqref{derivativeestimate} and condition \eqref{resolv-cond-gev} are equivalent for $C_0$-semigroups on Hilbert spaces~\cite{Wakaiki}, and hence Corollary \ref{gev-theta-sub} still holds for those semigroups on a Hilbert space whose Gevrey regularity is proved using \eqref{resolv-cond-gev}. This will indeed be the case in the examples considered in Sections~\ref{coupledplate} and~\ref{couplekirchhof-elect}.
\end{contrk}
%{\bf Remark~\ref{rk:gevrey-sem}. (continued):} %{\RG [we could eventually set up a new environment to have here a Remark with the same number as the previous one, but with the same characteristics of a Remark environment.]} 

{\begin{rk}\label{thmthetabound}
    As evident from the relations in~\eqref{rel-theta-Abound-relation} and the estimates in~\eqref{estfrctpw} and~\eqref{estfrctpw2}, respectively, Theorem~\ref{thmanalyticgib} and Corollary~\ref{gev-theta-sub} still hold when \(B\) is \emph{$(-A)^\theta$-bounded} (in the sense of Definition \ref{athetabound}).  However, the  argument fails  under $A$-bounded perturbations (in the sense of Definition \ref{Abound}), since in place of \eqref{estfrctpw}, we would have, for any $s>0$,
\begin{equation*}
\|\phi_n(s)\| = \|B e^{sA}(x_n - \bar{x})\| \le 2a \|e^{sA}\| + 2b\|Ae^{sA}\| \le 2Me^{-\omega s}\left[a + \frac{b}{s}\right],
%\quad \forall s>0,
\end{equation*}
which does not imply convergence of $\phi_n\rightarrow 0$ in $L^1(0,+\infty;X)$ as $n\rightarrow+\infty,$ because %bound $2Mbe^{-\omega s}/s$ on 
the right hand side is not  integrable on $(0,\infty)$.
\end{rk}

\begin{rk}
 
A related question of interest is whether the regularity of the given semigroup is preserved under relatively bounded perturbations of the semigroup generator of the type described in Definitions \ref{Abound}, \ref{theta-sub} and \ref{athetabound}. 
%We would like to mention that analyticity is closed under the above perturbations in Definitions \ref{Abound}, \ref{theta-sub} and \ref{athetabound}. 
A well known result~\cite[Section 3.2]{Pazy} ensures the closedness of analyticity for $A$-bounded perturbations in the sense of Definition \ref{Abound}---also referred to as relatively bounded perturbations in the literature (see \cite{kato,Pazy}).
That is, if $A$ is the infinitesimal generator of an analytic semigroup and $B$ is $A$-bounded, then $A + B$ is the infinitesimal generator of an analytic semigroup too. It follows readily from \eqref{rel-theta-Abound-relation} that this holds even for 
$\theta$-subordinate and $(-A)^{\theta}$-bounded perturbations.
% This can be easily verified from 
However, as noted in~\cite[Section 4]{Grab-las}, it is an open question whether Gevrey regularity is closed under relatively bounded perturbations of its generator. Nevertheless, a partial result in this direction is given in the same reference \cite[Lemma 4.2]{Grab-las}: if $A$ generates a differentiable semigroup on a Banach space $X$ satisfying assumption  \eqref{derivativeestimate} for some $\beta\ge 1$, and $B$ is $(-A)^\theta$-bounded for some $\theta<1/\beta$, %{\RG [of which type? Def. 1/2/3?]}
then $A+B$ generates a Gevrey semigroup of class $\delta$ for $t>0$, for every $\delta >\beta.$
\end{rk}

\section{Applications}\label{sec:applic}
In this section, we provide applications of the abstract results obtained in the previous section to various partial differential equations. 
In what follows, we adopt the following notations. The symbols $\lesssim$ and $\gtrsim$ denote inequalities that holds up to a positive multiplicative constant. The symbol $\rho(A)$ stands for the resolvent set of the operator~$A$. Time derivatives of first and second order of the function $u = u(x,t)$ are denoted by~$\partial_t u$ and $\partial_t^{2}u$, respectively, or equivalently by~$u_t$ and~$u_{tt}$. Finally, we denote by $\Omega \subset \mathbb{R}^d$, $d \geq 1$, a bounded domain in $\mathbb{R}^d$ with a sufficiently smooth boundary $\Gamma$. The notation $\|\cdot\|$ will be used to represent both the usual norm in the Hilbert space~$L^{2}(\Omega)$, as well as the Euclidean norm in $\mathbb{R}^d$.% unless otherwise specified. 

\subsection{A uniformly parabolic problem}\label{uniparsec}

 Let $A:D(A)\subset L^2(\Omega)\to L^2(\Omega)$ be the unbounded operator defined by
\begin{equation}\label{defoptA}
D(A) = H^2(\Omega)\cap H^1_0(\Omega)\, ,\qquad Au = \Delta u\quad \forall u\in D(A)\, .
\end{equation}
We analyze the stability of the parabolic problem
\begin{equation} \label{parbolic_eq}
\left\lbrace \begin{array}{ll}
\partial_t u = \Delta u  + b(x)\cdot \nabla u & \text{in } \Omega\times (0, +\infty)\,,\\
u =0 & \text{on } \Gamma\times (0, +\infty)\,,\\
u(0,\cdot) = u_0(\cdot)\in L^2(\Omega)\, ,
\end{array}\right.
\end{equation}
where $b:\Omega\to \mathbb{R}^d$ is a bounded measurable function of class ${C}^1$ on $\Omega$.  The goal is to consider the term $b(x)\cdot \nabla u$ as an unbounded perturbation of an exponentially stable system and determine appropriate conditions on the coefficient $b$ to be able to derive the exponential stability of~\eqref{parbolic_eq} from Theorem~\ref{thmanalyticgib}. Note that the operator $A$ in \eqref{defoptA} generates an analytic % $C_0$-
semigroup $\left(e^{tA}\right)_{t\ge 0}$ of compact operators on the Hilbert space $H = L^2(\Omega)$. % for every $t>0$. 
Moreover, the least eigenvalue $\lambda_\Omega$ of $-A$ is the Poincar\'e constant of $\Omega$, that is, the optimal positive constant such that
\begin{equation}\label{poincIneq}
\lambda_\Omega\int_\Omega u^2 \;\mathrm{d}x \le \int_\Omega |\nabla u|^2 \;\mathrm{d}x\qquad \forall u\in H^1_0(\Omega)\, .
\end{equation}
It is well known that the semigroup generated by $A$ is exponentially stable, since
$$
\|e^{tA}\|_{\mathcal{L}} \le e^{-\lambda_\Omega t}\qquad \forall t\ge 0\, ,
$$
and therefore condition $ii)$ of Theorem \ref{thmanalyticgib} holds. Suppose $b$ satisfies
\begin{equation}\label{Hb-1}
%\left\{
\begin{array}{ll}
\text{\textbf{H1)}} & \textbf{div} (b) \geq - 2 \lambda_\Omega\quad \text{a.e. in } \Omega\,,%(0,1), 
\\[1ex]
  \text{\textbf{H2)}} & \textbf{div} (b) > - 2 \lambda_\Omega\quad \text{a.e. in some open set } U\subseteq \Omega\, .%(0,1).
\end{array}
%\right.
\end{equation}

We introduce the operator $B: H^1_0(\Omega)\to L^2(\Omega)$ such that
$$
Bu(x) = b(x)\cdot\nabla u(x)\, ,\quad x\in \Omega\,,\qquad \forall u\in H^1_0(\Omega)\, .
$$

Thus, we can reformulate~\eqref{parbolic_eq} in the abstract form
\begin{equation} \label{abst_eq}
\left\lbrace \begin{array}{ll}
u'(t) = (A+B) u(t) & t\in (0, \infty)\,,\\
u(0) = u_0\in H\, . &
\end{array}\right.
\end{equation}
Note that $B$ is $1/2$-subordinate to $A$, since
\begin{multline*}
        \|Bu\| =  \Bigg(\int_\Omega \|b(x)\cdot \nabla u(x)\|^2 \;\mathrm{d}x \Bigg)^{1/2} \leq \|b\|_{L^\infty} \Bigg(\int_\Omega \|\nabla u\|^2 \;\mathrm{d}x\Bigg)^{1/2}\\
\leq\|b\|_{L^\infty}\Bigg(\int_\Omega u(-Au)\;\mathrm{d}x\Bigg)^{1/2}
\leq\|b\|_{L^\infty}\|(-A)u\|^{\tfrac{1}{2}}\|u\|^{\tfrac{1}{2}}
    \end{multline*}
for any $u\in D(A)$. In fact, the operator $B$ is also $(-A)^{1/2}$-bounded, since $D((-A)^{1/2}) = H^1_0(\O)\subseteq D(B)$ and, for every $u\in H^1_0(\O)$,

\begin{equation*}
\|Bu\| = \int_\O \|b(x)\cdot \nabla u\|^2 \;\mathrm{d}x \le \|b\|_{L^\infty}^2\int_\O\|\nabla u\|^2 \;\mathrm{d}x = d \|(-A)^{1/2} u\|\, .
\end{equation*}

Moreover, thanks to the Rellich's inclusion, $Be^{tA}$ is a compact operator from $H^1_0(\Omega)$ to $L^2(\Omega)$ for every $t>0$, so both conditions $i)$ and $iii)$ of Theorem \ref{thmanalyticgib} are satisfied.
Therefore, in order to apply Theorem \ref{thmanalyticgib}, we shall verify that $A+B$ generates an asymptotically stable semigroup. 
Owing to the La Salle's principle, this is equivalent to show that the elliptic boundary value problem
\begin{equation} \label{elpt_eq}
\left\lbrace \begin{array}{ll}
\Delta u + b(x)\cdot \nabla u = 0 & \text{in } \Omega\,,\\
u = 0 & \text{on } \Gamma\, ,
\end{array}\right.
\end{equation}
only admits the trivial solution $u\equiv 0$. For this purpose, let $u$ be the solution of~\eqref{elpt_eq}. Multiplying the first equation in \eqref{elpt_eq} by $u$ and integrating by parts, from equation \eqref{poincIneq} we deduce that
$$
\lambda_\Omega \int_\Omega u^2 \;\mathrm{d}x \le \int_\Omega \|\nabla u\|^2 \;\mathrm{d}x = - \frac{1}{2}\int_\Omega \textbf{div} (b) u^2 \;\mathrm{d}x\, ,
$$
so hypothesis \textbf{H1)} ensures that
$$
\int_\Omega (\lambda_\Omega + \frac{1}{2} \textbf{div}(b)) u^2 \;\mathrm{d}x = 0\, .
$$
Thus,
$$
(\lambda_\Omega + \frac{1}{2} \textbf{div}(b)) u^2 = 0\quad \text{for a. e.}\quad x\in \Omega\, .
$$
Owing to assumption \textbf{H2)}, we deduce that $u\equiv 0$ in $U$, that in turn implies $u\equiv 0$ in $\Omega$ by well-known unique continuation property for elliptic operators~\cite{Hormander63}.
Thus, under hypoteses \eqref{Hb-1} on the coefficient function $b$, we can apply Theorem~\ref{thmanalyticgib} to~\eqref{abst_eq} and conclude that the solution of~\eqref{parbolic_eq} decays exponentially in time.

%%%%%%%%%%%%%%%%%%%%%%%%%%%%%%%%%%%%%%%%%%%%%%%%%%%%%%%%%%%%%%%%%%%%%%%%%%%%%%%%%%%%%%%%%%%%%%%%%%
\subsection{A degenerate/singular parabolic problem}\label{degsingsec}

We begin this section by recalling the notion of degenerate diffusion coefficients~\cite[Chapter 1]{Fra-mug}.% as follows.
\begin{Def}[Weakly degenerate case]
A function $a$ is said to be \emph{weakly degenerate} at $x_{0}\in[0,1]$ if 
$a\in C([0,1])\cap C^{1}([0,1]\setminus\{x_{0}\})$, $a(x_{0})=0$, $a>0$ on 
$[0,1]\setminus\{x_{0}\}$ and there exists $K_{a}\in(0,1)$ such that 
\[
 (x-x_{0})a'(x)\le K_{a}\,a(x)\qquad \forall\, x\in[0,1]\setminus\{x_{0}\}.
\]
\end{Def}
\begin{Def}[Strongly degenerate case]
A function $a$ is said to be \emph{strongly degenerate} at $x_{0}\in[0,1]$ if 
$a\in C^{1}([0,1]\setminus\{x_{0}\})\cap W^{1,\infty}(0,1)$, $a(x_{0})=0$, 
$a>0$ on $[0,1]\setminus\{x_{0}\}$ and there exists $K_{a}\in[1,2)$ such that 
\[
 (x-x_{0})a'(x)\le K_{a}\,a(x)\qquad \forall\, x\in[0,1]\setminus\{x_{0}\}.
\]
\end{Def}
In this section, we consider problems with degeneracy at the boundary. For simplicity, we will consider $x_{0}=0$ and the degenerate diffusion coefficient of the form $a(x) = x^\alpha$ for $\alpha\in [0,2)$, which provides a prototypical example of either weakly---$\alpha \in (0,1)$---or strongly---$\alpha\in [1,2)$---degenerate diffusions. %{, whereas the case of interior degeneracy is treated in Section \ref{intdegSection}.}}
%%%%%%%%%%%%%%%%%%%%%%%%%%%%%%%%%%%%%%%%%%%%%%%%%%%%%%%%%%%%%%%%%%%%%%%%%%%%%%%%%%%%%%%%%%%%%%%%%%
We consider the one dimensional degenerate/singular parabolic problem
\begin{equation} \label{full_parbolic_eq}
\left\lbrace \begin{array}{ll}
\partial_t u = (x^\a u_x)_x  + b(x) u_x + \frac{\mu}{x^\gamma} u & \text{in } (0,1)\times (0, +\infty)\,,\\
u(1,t) = 0 & \text{on } (0, +\infty)\,,\\
\left\lbrace \begin{array}{ll}
u(0,t) = 0 & \text{if } 0\le \a < 1 \,,\\
x^\a u_x(0,t) = 0 & \text{if } 1 \leq \a < 2 \,,
\end{array}\right.
& \text{on } (0, +\infty)\,,\\
u(0,\cdot) = u_0(\cdot)\in L^2(0,1)\, ,
\end{array}\right.
\end{equation}
with constants $\mu\ge 0$ and $\gamma >0$, and suitable assumptions on the function $b$ to be specified later.
Given $\a\in [0,2)$, we introduce the Hilbert space
\begin{equation}
H^1_\a (0,1) := \{u\in L^2(0,1)\cap H^1(\varepsilon,1): x^{\a/2}u_x\in L^2(0,1)\}\, ,
\end{equation}
where $H^1(\varepsilon,1)$ is the space of weakly differentiable functions in every open subset $(\varepsilon, 1)$ of $(0,1)$ for all $\varepsilon > 0$, endowed with the scalar product
\begin{equation}
(u,v)_{H^1_\a (0,1)} := \int_0^1 (uv + x^\a u_x v_x) \;\mathrm{d}x\qquad \forall u,\, v\in H^1_\a (0,1)\, .
\end{equation}
For all $u\in H^1_\a (0,1)$, the trace in $x=1$ is clearly well-defined, whereas the trace in $x=0$ makes sense only in the case $0\le \a < 1$~\cite{CanMarVan05}. This leads us to define
\begin{eqnarray}
&& H^1_{\a,0} (0,1) := \{u\in H^1_\a (0,1) : u(0) = u(1) = 0\}\quad \text{if $0\le \a < 1$}\, , \\
\noalign{\medskip}
&& H^1_{\a,0} (0,1) := \{u\in H^1_\a (0,1) : u(1) = 0\}\quad \text{if $1\le \a < 2$}\, .
\end{eqnarray}%{\RG [We shall start the section with this equation, and then introduce the appropriate functional framework to study it, followed by Hardy's inequality.]}
We will need extensively the following Hardy's inequality~\cite{Vancost11}: for every $z\in H^1_{\a,0}(0,1)$,
\begin{equation}\label{Hardyineq}
\lambda(\a)\int_0^1 \frac{z^2}{x^{2-\a}}\;\mathrm{d}x \le \int_0^1 x^\a z_x^2 \;\mathrm{d}x\, ,
\end{equation}
where $\lambda(\a) := (1-\a)^2/4$. 
This implies that there exists $\mu(\a) >0$ such that
\begin{equation}\label{HP}
\mu(\a) \int_0^1 z^2 \;\mathrm{d}x \le \int_0^1 x^\a z_x^2 \;\mathrm{d}x\qquad \forall z\in H^1_{\a,0}(0,1)\, ,
\end{equation}
which in turn ensures the equivalence of the norm
$$
\|u\|_{H^1_{\a,0}(0,1)}^2 := \int_0^1 x^\a u_x^2 \;\mathrm{d}x\qquad \forall z\in H^1_{\a,0}(0,1)\, .
$$
For simplicity, we study the purely degenerate case in Subsection \ref{deg_case} ($\mu = 0$) and the degenerate/singular case ($\mu\neq 0$) in Subsection~\ref{sing_case}.

%%%%%%%%%%%%%%%%%%%%%%%%%%%%%%%%%%%%%%%%%%%%%%%%%%%%%%%%%%%%%%%%%%%%%%%%%%%%%%%%%%%%%%%%%%%%%%%%%%
\subsubsection{The degenerate case}\label{deg_case}
%%%%%%%%%%%%%%%%%%%%%%%%%%%%%%%%%%%%%%%%%%%%%%%%%%%%%%%%%%%%%%%%%%%%%%%%%%%%%%%%%%%%%%%%%%%%%%%%%%

We consider the one dimensional degenerate parabolic problem
\begin{equation} \label{deg_parbolic_eq}
\left\lbrace \begin{array}{ll}
\partial_t u = (x^\a u_x)_x  + b(x) u_x & \text{in } (0,1)\times (0, +\infty)\,,\\[0.5ex]
u(1,t) = 0 & \text{on } (0, +\infty)\,,\\
\left\lbrace \begin{array}{ll}
u(0,t) = 0 & \text{if } 0\le \a < 1 \,,\\
(x^\a u_x)(0,t) = 0 & \text{if } 1 \leq \a < 2 \,,
\end{array}\right.
& \text{on } (0, +\infty)\,,\\
u(0,\cdot) = u_0(\cdot)\in L^2(0,1)\, ,
\end{array}\right.
\end{equation}
where $b:(0,1)\to \R^{+}\cup\{0\}$ is a bounded measurable function of class $C^1$.
We introduce the $1$-D degenerate operator
\begin{equation}
\begin{array}{c}
D(A) := \{u\in H^1_{\a,0} (0,1)\cap H^2(\varepsilon,1) : (x^\a u_x)_x\in L^2(0,1)\}\, , \\
\noalign{\medskip}
A u = (x^\a u_x)_x\qquad \forall u\in D(A)\, .
\end{array}
\end{equation}
We point out that, if $0\le \a < 1$, every $u\in D(A)$ satisfies Dirichlet boundary conditions $u(0) = u(1) = 0$, while, if $1\le \a < 2$, $u\in D(A)$ implies the Neumann-type boundary condition $(x^\a u_x)(0) = 0$ at $x = 0$ and Dirichlet condition $u(1) = 0$ at $x=1$ (see \cite{CanMarVan05} for a detailed description of the domain $D(A)$. We refer to \cite{CanRocVan} for the analysis of the operator $(A,D(A))$ in higher dimension). 
The operator \((A,D(A))\) generates an analytic semigroup \(\left(e^{tA}\right)_{t\ge 0}\), which is exponentially stable, i.e.,
\[
\|e^{tA}\|_{\mathcal{L}} \le e^{-\mu(\alpha)t} \qquad \forall\, t \ge 0,
\]
where \(\mu(\alpha)\) denotes the constant from Hardy's inequality \eqref{HP}. 
We now define the operator $B: H^1_\a (0,1)\to L^2(0,1)$ by
\begin{equation}\label{perB}
B u(x) = b(x) u_x(x)\qquad \forall u\in  H^1_\a (0,1) \, ,\ x\in (0,1)\, ,
\end{equation}
and suppose furthermore that
\begin{equation}\label{Hb0}
\left\{
\begin{array}{ll}
  b' &\geq - 2 \mu(\alpha)\quad \text{a.e. in } (0,1), \\[1ex]
  b' &> - 2 \mu(\alpha)\quad \text{a.e. in some open set } U\subset(0,1).
\end{array}
\right.
\end{equation}
Then $D(A)\subset D((-A)^{1/2}) = H^1_{\a,0}(0,1) \subset D(B)$. Moreover, in order for $B$ to satisfy the $\theta$-subordinate condition with $\theta=1/2$, that is,
$$\int_0^1b^2(x)u^2_x \,\mathrm{d}x\lesssim \int_0^1x^\alpha u^2_x \,\mathrm{d}x=\int_0^1u(-Au)\,\mathrm{d}x\leq\|Au\|\|u\|,\qquad\forall u\in D(A),$$
%$\|Bu\|\le c\|(-A)^{1/2}u\|$ 
we shall assume that
\begin{equation}\label{condonb}
|b(x)|\le cx^{\a/2} \quad\text{ for some constant $c>0$} \,.
\end{equation}
This requires $b$ to be degenerate at the point $x=0$. If $b$ has the particular form $b(x) = x^\beta$ for some positive $\beta\in \R$, condition \eqref{condonb} reduces to $\beta \ge \a /2$. %{\RG [Please double check whether we need the same structure of $b$, i.e., to vanish at zero with a suitable exponent, even if we rely on $\theta$-subordination. In any case, the assumptions on $b$ must be reformulated in a consisten way. As it is, we later refer to~\eqref{condonb} but that assumption does not include the structure of $b$ as well.]}\textcolor{magenta}{ [If we rely on $\theta$ subordinate, then also we require same structure on b, for example for $\theta$-subordination it will look like:
%$$\|Bu\|=\left(\int_0^1b^2(x)u^2_x \,\mathrm{d}x\right)^{1/2}\leq c\left(\int_0^1x^\alpha u^2_x \,\mathrm{d}x\right)^{1/2}=c\left(\int_0^1u(-Au)\,\mathrm{d}x\right)^{1/2}\leq c\|Au\|^{1/2}\|u\|^{1/2}$$ Should we use $\theta$-subordinate or leave it as it is. What do you suggest?]} {\RG [It does not seem to be a real advantage in using one or the other properties here.]}\\
Moreover, $B e^{tA}$ is a compact operator for all $t>0$, thanks to Rellich's embedding and the compactness of the semigroup $\left(e^{tA}\right)_{t\ge 0}$ in $H = L^2(\O)$ for all $t > 0$.
Finally, in order to prove the exponential stability of the semigroup generated by $A+B$ (and thus of the solution to system \eqref{deg_parbolic_eq}) it suffices to check that the semigroup $\left(e^{t(A+B)}\right)_{t\geq 0}$ is strongly stable. The La Salle's principle ensures that this is the case provided that the elliptic boundary value problem
\begin{equation} \label{deg_elliptic_eq}
\left\lbrace \begin{array}{ll}
(x^\a u_x)_x  + b(x) u_x = 0 & \text{in } (0,1) \,,\\[0.5ex]
u(1) = 0 & \\
\left\lbrace \begin{array}{ll}
u(0) = 0 & \text{if } 0\le \a < 1 \,,\\
(x^\a u_x)(0) = 0 & \text{if } 1 \leq \a < 2 \,,
\end{array}\right.
& \\
u(\cdot) = u_0(\cdot)\in L^2(0,1)\, ,
\end{array}\right.
\end{equation}
admits the only solution $u=0$. Indeed, multiplying the first equation by $u$, integrating by parts, and thanks to \eqref{HP}, we have 
\begin{multline}\label{234p}
\mu(\alpha)\int_0^1 u^2 \,\mathrm{d}x \le \int_0^1 x^\alpha u_x^2 \,\mathrm{d}x = \int_0^1 b(x)\,u_x\,u \,\mathrm{d}x  \\ 
= \frac{1}{2}\big( b(1)\,u(1)^2 - b(0)\,u(0)^2 \big) 
   - \frac{1}{2}\int_0^1 b'(x)\,u^2 \,\mathrm{d}x .
\end{multline}
The first boundary term $b(1)u^2(1)$ vanishes due to the imposed boundary condition on $u(1)$ for all $\alpha\in[0,2)$, while the term $b(0)u^2(0)$ certainly vanishes when $0\le \alpha<1$. For $1\le \alpha < 2$, we argue as follows. Notice that, for a.e. $x\in(0,1)$, we have $$u(x)=-\int_x^1 u'(s)\,ds,$$
which, thanks to Cauchy-Schwartz inequality, implies 
$$
|u(x)|^2
\leq \left(\int_x^1 s^{-\alpha}\,ds\right)
     \left(\int_x^1 s^\alpha |u'(s)|^2\,ds\right)
\leq \left(\int_x^1 s^{-\alpha}\,ds\right)\,\|u\|_{H^1_\alpha(0,1)}^2.
$$
Therefore, from \eqref{condonb}, we obtain $$|b(x)|\,|u(x)|^2
\leq c\,x^{\alpha/2}\left(\int_x^1 s^{-\alpha}\,ds\right)\,\|u\|_{H^1_\alpha(0,1)}^2 .$$ Thus
\[
|b(x)|\,|u(x)|^2
\,\leq\,
\begin{cases}
C\,x^{1/2}|\ln x|\,\|u\|_{H^1_\alpha(0,1)}^2 \,\xrightarrow[x\to 0^+]{}\, 0,
& \text{if } \alpha = 1, \\[2ex]
C\,
x^{1-\alpha/2}\left(\frac{x^{\alpha-1}-1}{1-\alpha}\right)\,\|u\|_{H^1_\alpha(0,1)}^2 \,\xrightarrow[x\to 0^+]{}\, 0,
& \text{if } 1<\alpha<2,
\end{cases}
\]
which, ensures that $b(0)u^2(0)$ is equal to zero for $1\le \alpha<2$ as well. Hence, thanks to the conditions~\eqref{Hb0} on the function $b$, the estimate \eqref{234p} implies that $u\equiv 0$ in $U$. Then, by the unique continuation result \cite{Hormander63}, we conclude that $u\equiv 0$ in $(0,1)$. Thus, system~\eqref{deg_parbolic_eq} with coefficient $b(x)$ satisfying~\eqref{Hb0} and~\eqref{condonb} is exponentially stable.

%%%%%%%%%%%%%%%%%%%%%%%%%%%%%%%%%%%%%%%%%%%%%%%%%%%%%%%%%%%%%%%%%%%%%%%%%%%%%%%%%%%%%%%%%%%%%%%%%%
\subsubsection{The degenerate/singular case}\label{sing_case}
%%%%%%%%%%%%%%%%%%%%%%%%%%%%%%%%%%%%%%%%%%%%%%%%%%%%%%%%%%%%%%%%%%%%%%%%%%%%%%%%%%%%%%%%%%%%%%%%%%

%{\RG [Comments for this section:\\[1ex]Apply the main theorem to the operator of the degenerate/singular equation in Vancostenoble 2011.\\[1ex]
%Per la sezione successiva devo capire se il semigruppo degenerato dall'operatore $A$ con degenerazione e potenziale singolare è compatto in $H$ o no! (vedi Vancostenoble '11)Verify whether the main theorem applies to the operator of the degenerate equation with internal degeneracy from Goldstein1-2-Fragnelli-Romanelli.
%), devo chiarire se il generatore è negative o nonpositive, o bisogna di negative per poter avere stabilità esponenziale del semigruppo non perturbato!!)]}
Consider the one dimensional degenerate/singular parabolic problem
\begin{equation} \label{sing_parbolic_eq}
\left\lbrace \begin{array}{ll}
\partial_t u = (x^\a u_x)_x  + b(x) u_x+\frac{\mu}{x^\gamma} u & \text{in } (0,1)\times (0, +\infty)\,,\\[0.5ex]
u(1,t) = 0 & \text{on } (0, +\infty)\,,\\
\left\lbrace \begin{array}{ll}
u(0,t) = 0 & \text{if } 0\le \a < 1 \,,\\
x^\a u_x(0,t) = 0 & \text{if } 1 \leq \a < 2 \,,
\end{array}\right.
& \text{on } (0, +\infty)\,,\\
u(0,\cdot) = u_0(\cdot)\in L^2(0,1)\, ,
\end{array}\right.
\end{equation}
where  $\a \in [0,2)$, $0<\mu$, $0<\gamma<2-\alpha $ and $b:(0,1)\to \R^{+}\cup\{0\}$ is a bounded measurable function of class $C^1$.
In order to recast this equation in an abstract formulation, we define the operator $(A,D(A))$ as
\begin{equation}
\begin{array}{c}
D(A) := \{u\in H^1_{\a,0} (0,1)\cap H^2(\varepsilon,1) : (x^\a u_x)_x+\frac{\mu}{x^\gamma} u\in L^2(0,1)\}\, , \\
\noalign{\medskip}
A u = (x^\a u_x)_x+\frac{\mu}{x^\gamma} u\qquad \forall u\in D(A)\, .
\end{array}
\end{equation}

\noindent
Note that, for $\gamma \leq 0$, the coefficient $\mu/x^\gamma$ is in $L^\infty(0,1)$ and thus the exponential stability of system~\eqref{sing_parbolic_eq} can be easily derived from the purely degenerate case studied in Section~\ref{deg_case}.
%Furthermore, for $\alpha = 1$, we have $\sqrt{x}\, u_x \in L^2(0,1)$ does not imply $\frac{u}{\sqrt{x}} \in L^2(0,1)$. Hence, we consider $0 < \gamma < 2 - \alpha$. {\RG [Revise the argument of the sentence starting from Furthermore.]}

We recall the following improved Hardy-Poincaré's inequality (see \cite{Vancost11}): for every $\a \in [0,2)$, $n >0$ and $\gamma < 2 - \a$ there exists $C_0 = C_0(\a,\gamma,n)>0$ such that, for every $z\in H^1_{\a,0}(0,1)$,
\begin{equation*}\label{impH-P}
n\int_0^1 \frac{z^2}{x^\gamma} \;\mathrm{d}x + \lambda(\a)\int_0^1 \frac{z^2}{x^{2-\a}}\;\mathrm{d}x \le \int_0^1 x^\a z_x^2 \;\mathrm{d}x + C_0 \int_0^1 z^2 \;\mathrm{d}x\, ,
\end{equation*}
where $\lambda(\a) := (1-\a)^2/4$. Thus,
$$
    \frac{\mu}{n}\int_0^1x^{\alpha}u^2_x \;\mathrm{d}x-\mu\int_0^1 \frac{u^2}{x^\gamma} \;\mathrm{d}x \geq -\frac{\mu C_0}{n}\int_0^1u^2\;\mathrm{d}x +\frac{\mu \lambda(\alpha)}{n}\int_0^1\frac{u^2}{x^{2-\alpha}}\;\mathrm{d}x,
    $$
    which together with~\eqref{HP} implies that
\begin{align}\label{ine}
    \int_0^1x^{\alpha}u^2_x \;\mathrm{d}x-\mu\int_0^1 \frac{u^2}{x^\gamma} \;\mathrm{d}x &\geq \left(1- \frac{\mu}{n}\right)\int_0^1x^{\alpha}u^2_x \;\mathrm{d}x-\frac{\mu C_0}{n}\int_0^1u^2\;\mathrm{d}x\nonumber\\
    &\geq \left\{\mu(\alpha)\left(1- \frac{\mu}{n}\right)-\frac{\mu C_0}{n}\right\}\int_0^1u^2\;\mathrm{d}x.
\end{align}
Thus, for \(0<\mu<\frac{n\mu(\alpha)}{C_0+\mu(\alpha)}\), there exists $$\tau(\mu,\alpha, n,\gamma) :=\left\{\mu(\alpha)\left(1- \frac{\mu}{n}\right)-\frac{\mu C_0}{n}\right\}>0 $$ such that $(A,D(A))$ generates an exponentially stable analytic semigroup $\left(e^{tA}\right)_{t\ge 0}$ with
$$
\|e^{tA}\|_{\mathcal{L}} \le e^{-\tau(\mu,\alpha, n,\gamma)t}\quad \forall t\ge 0\, .
$$
Now, we consider the same perturbation \eqref{perB} under the following conditions.
\begin{equation}\label{Hb1}
\left\{
\begin{array}{ll}
  b'(x) &\geq - 2 \tau(\mu,\alpha, n,\gamma)\quad\quad \text{a.e. in } (0,1), \\[1ex]
  b'(x) &> - 2 \tau(\mu,\alpha, n,\gamma)\quad\quad \text{a.e. in some open set } U\subset(0,1)\\[1ex]
  |b(x)| &\leq  c\sqrt{\frac{ \tau(\mu,\alpha, n,\gamma)}{\mu(\alpha)}} x^{\alpha/2} \quad\text{for some constant}\,\,c>0.
\end{array}
\right.
\end{equation}
%{\RG [Why do we need to explicitly write the constant in the third equation, even though we're using $\lesssim$? Similar comment for the estimate in the following equation.]}\textcolor{magenta}{[We need that because we are using \eqref{ine}, I understand that $c$ will be adjusted accordingly but it make more sense to demonstrate in precise way how we are getting below inequality. I think its fine to write like that, as it avoids mentioning extra $c>0$ unnecessarily OR we can just add extra $c$ with $\leq$, instead of $\lesssim$]}
We then observe that \( D(A)\subset D((-A)^{1/2}) = H^1_{\alpha,0}(0,1) \subset D(B) \). In addition, we have  
\begin{multline*}
    \|Bu\|^2
    = \int_0^1 b^2(x)u_x^2\,\mathrm{d}x 
    \;\leq c^2\; \frac{\tau(\mu,\alpha,n,\gamma)}{\mu(\alpha)}\int_0^1 x^{\alpha}u_x^2\,\mathrm{d}x\\
    \leq c^2\; \left(\left(1- \frac{\mu}{n}\right)-\frac{\mu C_0}{n\mu(\alpha)}\right)\int_0^1 x^{\alpha}u_x^2\,\mathrm{d}x
    \lesssim\; \int_0^1 \left(x^{\alpha}u_x^2 - \mu\frac{u^2}{x^{\gamma}}\right)\mathrm{d}x\\
    \lesssim\; \int_0^1 u(-Au)\,\mathrm{d}x
    \lesssim\; \|(-A)u\|\,\|u\|\qquad\forall u\in D(A),
\end{multline*}
where we have used \eqref{HP}, \eqref{ine}, and the Cauchy-Schwarz inequality. Hence, \(B\) is \(1/2\)-subordinate to \(A\).

Furthermore, by virtue of Rellich’s compactness theorem, it is easy to see that \( B e^{tA} \) is a compact operator. By Theorem \ref{thmanalyticgib}, we claim that the semigroup generated by \( A + B \) is exponentially stable, provided the semigroup \( e^{t(A + B)} \) is strongly stable. 
In order to prove strong stability, we apply La Salle's principle again to show that the elliptic boundary value problem
\begin{equation} \label{deg_singelliptic_eq}
\left\lbrace \begin{array}{ll}
(x^\a u_x)_x  + b(x) u_x +\frac{\mu}{x^\gamma} u = 0 & \text{in } (0,1) \,,\\[0.5ex]
u(1) = 0 & \\
\left\lbrace \begin{array}{ll}
u(0) = 0 & \text{if } 0\le \a < 1 \,,\\
(x^\a u_x)(0) = 0 & \text{if } 1 \le \a < 2 \,,
\end{array}\right.
& \\
u(\cdot) = u_0(\cdot)\in L^2(0,1)\, ,
\end{array}\right.
\end{equation}
has only solution $u=0$ as below.
Multiplying the first equation in \eqref{deg_singelliptic_eq} by $u$, integrating by parts, and using \eqref{ine}, we obtain
$$
\tau(\mu,\alpha,n,\gamma)\int_0^1 u^2 \,\mathrm{d}x 
\le \int_0^1 \left( x^\alpha u_x^2 - \mu \frac{u^2}{x^\gamma} \right)\,\mathrm{d}x 
= \int_0^1 -\frac{1}{2} b'(x)u^2 \,\mathrm{d}x,
$$
where the boundary terms are zero by similar arguments as in the purely degenerate case in the previous section. The above estimate, together with the conditions \eqref{Hb1} and the unique continuation result \cite{Hormander63}, implies that $u \equiv 0$ in $(0,1)$. Hence, we conclude that system~\eqref{sing_parbolic_eq} with the coefficient \(b(x)\) satisfying \eqref{Hb1} is exponentially stable.}

\subsection{Coupled plate dynamics}\label{coupledplate}

Inspired by the works \cite{ChenTriggiani1989, Chen-Triggiani1990}, the authors of \cite{Ammari-Shel-Tebou-2021} considered the following coupled abstract elastic system, {in which the operator matrix defining the damping mechanism is degenerate, in the sense that the associated damping matrix is singular of rank one:}
\begin{equation}\label{degcase}
\begin{cases}
y_{tt} + aAy + \gamma A^{\theta}(y_t + z_t) = 0, & t>0,\\[1ex]
z_{tt} + bAz + \gamma A^{\theta}(y_t + z_t) = 0, & t>0,
\end{cases}
\end{equation}
where $A$ is a positive self-adjoint operator on a Hilbert space $H$, and $a,b,\gamma>0$. Using resolvent estimates of the form~\eqref{resolv-cond-gev}, they proved that for $a\neq b$, the semigroup generated by the system \eqref{degcase} is of Gevrey class $\delta>1/(2\theta)$ for $\theta\in(0,1/4]$, and of class $\delta>(1+2\theta)/(3\theta)$ for $\theta\in(1/4,1/2]$, differentiable but not analytic for $\theta\in(1/2,1]$, and exponentially stable for $\theta\in[0,1]$. Extending their analysis to the non-degenerate case, in \cite{Ammari-Shel-Tebou-2023} they studied the following abstract problem:
\begin{equation}\label{nondegcase}
\begin{cases}
y_{tt} + A_1 y + \alpha B_1 y_t + \beta B_1 z_t = 0, & t>0,\\[1ex]
z_{tt} + A_2 z + \beta B_1 y_t + \gamma B_2 z_t = 0, & t>0,
\end{cases}
\end{equation}
where $\alpha,\beta,\gamma>0$ and $A_j,B_j$ $(j=1,2)$ are positive self-adjoint operators such that for some  constants $\alpha_0,\alpha_1,\alpha_2,\beta_1,\beta_2>0$ and $\mu,\theta\in(0,1]$, we have 
\[
B_1 \le \alpha_0 B_2, \quad \beta_1 A_2^{\theta} \le B_2 \le \beta_2 A_2^{\theta}, \quad \beta_1 A_2^{\theta} \le B_2 \le \beta_2 A_2^{\theta}.
\]
In this case, the corresponding semigroup is analytic whenever $1/2 \le \mu,\theta \le 1$, and of Gevrey class $\delta\geq 1/(2\min(\mu,\theta))$ when $\min(\mu,\theta)\in(0,1/2)$. However, exponential stability is not discussed.

Motivated by these problems, we consider a coupled plate system with different fractional damping orders, where the fractional power in the self-damping term is higher than that in the cross-damping term. Precisely, we study
\begin{equation}\label{plateeq}
    \left\{
\begin{array}{ll}
\partial_t^2 u +\Delta^2 u + \Delta^{2\alpha} \partial_t u + \beta \Delta^{2\theta} \partial_t v = 0, & \text{in } \Omega \times (0,\infty), \\[1ex]
\partial_t^2 v +\Delta^2 v  + \Delta^{2\alpha} \partial_t v +\beta \Delta^{2\theta} \partial_t u = 0, & \text{in } \Omega \times (0,\infty), \\[1ex]
u = 0, \, \partial_\textbf{n}u=0,\, v = 0,\, \partial_\textbf{n}v=0, & \text{on } \Gamma \times (0,\infty),
\end{array}
\right.
\end{equation}
%{\RG [Can~\eqref{plateeq} be rephrased as either~\eqref{degcase} or~\eqref{nondegcase}? We shall make clearer the connection between this system and the previous ones.]}
where $0 < \alpha \le 1$, $0 < \theta <1/2$, $\theta<\alpha$, $\beta \in \mathbb{R}\backslash\{0\}$, and $\partial_\textbf{n}$ denotes the unit normal derivative on $\Gamma$. {In contrast with the degenerate model \eqref{degcase}, where the damping and coupling are governed by the same operator, and the non-degenerate model
\eqref{nondegcase}, in which the matrix governing the damping and coupling mechanism is not symmetric, the present system exhibits a symmetric fractional
damping structure with different fractional orders \((\theta<\alpha)\). Note that if \(\theta=\alpha\) is chosen, then system \eqref{plateeq} becomes a
particular case of system \eqref{nondegcase} with \(\alpha=\gamma=1\),
\(A_1=A_2=\Delta^2\), and $B_1=B_2=\Delta^{2\alpha}$, and it also corresponds
to system~\eqref{degcase} when \(a=b=\gamma=\beta=1\) and \(A=\Delta^2\).} %{\RG [Clarify the previous sentence.]} We demonstrate strong stability, and prove that the system generates an exponentially stable semigroup on the associated energy space.

Define $U := (u, u_t, v, v_t)^\top$ %{\RG [prefer $z^\top$ rather than $z^T$ for the transpose.]} 
and the space $\mathcal{H}
=
\left(\big[H^4(\Omega)\cap H^2_0(\Omega)\big] \times L^2(\Omega)\right)^2,$ on $\C$, equipped with the norm
\[
\|U\|^2_{\mathcal{H}}
=\|\Delta u\|^2+
\|u_t\|^2+
\|\Delta v\|^2+\|v_t\|^2.\]
%{\RG [Is $\|\,\cdot\,\|$ in the RHS the $L^2$-norm? We shall add this notation at the beginning of the Applications section, if that's used throughout the section.\\We shall also introduce the notations $\partial_t u$ and $\partial_t^2 u$ for time derivatives of the function $u$, as well as $u_t$ and $u_{tt}$.]\\}
Then system \eqref{plateeq} can be written as
\[
\partial_t U = (A  + B) U,
\]where
\[
A
=
\begin{pmatrix}\,\,\boxed{
\begin{matrix}
0 & I  \\
-\Delta^2 & - \Delta^{2\alpha}
\end{matrix}
}
& 0 \\
0 &
\boxed{
\begin{matrix}
0 & I \\
-\Delta^2 & - \Delta^{2\alpha}
\end{matrix}
}\,\,\end{pmatrix},\]\[
B
=
\begin{pmatrix}
0 &
\boxed{
\begin{matrix}
0 & 0 \\
0 & - \beta \Delta^{2\theta}
\end{matrix}
}\,\,
\\\,\,
\boxed{
\begin{matrix}
0 & 0 \\
0 & - \beta \Delta^{2\theta}
\end{matrix}
}
& 0
\end{pmatrix},
\]
and
\[
D(A)
=
\left\{
U  \in \big[ H^4(\Omega) \cap H^2_0(\Omega)\big]^4
:
\begin{array}{l}
\Delta^2 u + \Delta^{2\alpha} u_t \in L^2(\Omega),\\
\Delta^2 v + \Delta^{2\alpha} v_t \in L^2(\Omega).
\end{array}
\right\},
\]
\[
D(B)
=
\left\{
U \in \mathcal{H}
:
u_t \in D(\Delta^{2\theta}), \;
v_t \in D(\Delta^{2\theta})
\right\}.
\]
Clearly, {$D(B) \supset D(A)$}, %{\RG [Here actually we state that $D(A)$ and $D(B)$ are related! Shouldn't we notice this right after introducing the operators $A$ and $B$ with their domains?]}
and given $\theta \in (0,1/2)$, we have
\begin{align*}
\|BU\|^2_{\mathcal{H}}
&=|\beta|^2\big[\|\Delta^{2\theta} u_t\|^2 + \|\Delta^{2\theta} v_t\|^2\big] \\
&\lesssim |\beta|^2\big[\|\Delta u_t\|^{4\theta} \|u_t\|^{2(1-2\theta)} + \|\Delta v_t\|^{4\theta} \|v_t\|^{2(1-2\theta)}\big] \\
&\lesssim |\beta|^2 \big[ \|\Delta u_t\|^{4\theta} + \|\Delta v_t\|^{4\theta} \big] \|U\|^{2(1-2\theta)}_{\mathcal{H}} \\
&\lesssim |\beta|^2 \big[ \|\Delta u_t\|^{2} + \|\Delta v_t\|^{2} \big]^{2\theta} \|U\|^{2(1-2\theta)}_{\mathcal{H}} \\
&\lesssim |\beta|^2 \|A U\|^{4\theta}_{\mathcal{H}} \|U\|^{2(1-2\theta)}_{\mathcal{H}},
\end{align*}
where we have used the interpolation inequality and Jensen's inequality (for concave functions). Hence, $B$ is a $2\theta$-subordinate perturbation of $A$.

Note that, using the Balakrishnan representation formula for the fractional power of the positive operators (see \cite[Chapter~4]{Lunardi}  or \cite[Section~2.6]{Pazy}), {for $\theta<\alpha$}, we have the following spectral estimate:
\begin{equation}\label{fracPoinSpectrInq}
    \|(-\Delta)^\theta x\|\leq \lambda_{\Omega}^{\theta-\alpha}\|(-\Delta)^\alpha x\|, \,\,\qquad \forall x\in D\bigl((-\Delta)^\alpha)\bigr),
\end{equation}
%{\RG [for which values of $\theta$ and $\alpha$? The same intervals mentioned after~\eqref{plateeq}?\\By the way, too many numbered equations. Make sure that only equations effectively referred to in the text are numbered, and the other can be unnumbered. Since we will end up with many numbered equations anyway, please check how to number them according to the section they are in, e.g., equation (3.10) rather than equation (51).]\\}
where $\lambda_{\Omega}$ is the Poincaré constant, i.e., the first Dirichlet eigenvalue of $-\Delta$, as in \eqref{poincIneq}. In the sequel, we restrict to $|\beta| < \lambda_{\Omega}^{2(\alpha-\theta)}$. Then, it is easy to see that
\begin{multline}\label{dissipa}
\Re\langle(A+B)U,U\rangle_{\mathcal{H}}
= -\|(-\Delta)^\alpha u_t\|^2 - \|(-\Delta)^\alpha v_t\|^2
- 2\beta\,\Re\langle(-\Delta)^\theta u_t,\,(-\Delta)^\theta v_t\rangle\\
\le -\|(-\Delta)^\alpha u_t\|^2 - \|(-\Delta)^\alpha v_t\|^2
+ 2|\beta|\,\|(-\Delta)^\theta u_t\|\,\|(-\Delta)^\theta v_t\|\\
\le -\|(-\Delta)^\alpha u_t\|^2 - \|(-\Delta)^\alpha v_t\|^2
+ |\beta|\bigl(\|(-\Delta)^\theta u_t\|^2+\|(-\Delta)^\theta v_t\|^2\bigr)\\
\le -\bigl(1-|\beta|\lambda_{\Omega}^{2(\theta-\alpha)}\bigr)
\bigl[\|(-\Delta)^\alpha u_t\|^2+\|(-\Delta)^\alpha v_t\|^2\bigr]\le 0,
\end{multline}%{\RG [Math comment: here you are skipping some details to get the first identity. I agree with that, I imagine that's straightforward. But I haven't double checked the computations, please do so one more time.\\Notation comment: Have you introduced, here or earlier in the text, the notation for the real part of a complex number?\\Latex comment: I usually prefer \{multline*\} rather than \{align*\} environments. \{multline*\} has a less rigid structure, which would fit the third inequality in one line rather than splitting it in two lines.\\Other tex comment: make sure to space out terms like $|\beta|\|(-\Delta)^\theta u_t\|^2$ with a space \verb|\,| in between. In this case, $|\beta|\,\|(-\Delta)^\theta u_t\|^2$.]\\}
where we applied the Cauchy-Schwarz inequality, Young's inequality, and estimate \eqref{fracPoinSpectrInq}. By a density argument, % {\RG [density of which space in which other one and with respect to which norm?]}, 
$A+B$ is closed in $\mathcal{H}$ and hence, by the Lumer-Phillips theorem,  it generates a strongly continuous semigroup of contractions on $\mathcal{H}$. For $\Phi:=(\phi_1,\phi_2,\phi_3,\phi_4)\in\mathcal{H}$, we consider the resolvent problem $i\lambda U-(A+B)U=\Phi\,,\lambda\in\R,$ which reads as follows:
 %{\RG[What's the relation between $D(A)$ and $D(B)$? To which space does $U$ belong? Specify also that $\Phi \in\mathcal{H}$.]\\}
\begin{align*}
    i\lambda u-u_t=\phi_1 &\in H^4(\Omega)\cap H^2_0(\Omega)\\
    i\lambda u_t+\Delta^2u+\Delta^{2\alpha}u_t+\beta\Delta^{2\theta}v_t=\phi_2&\in L^2(\Omega)\\
    i\lambda v-v_t = \phi_3&\in H^4(\Omega)\cap H^2_0(\Omega)\\
    i\lambda v_t+\beta\Delta^{2\theta}u_t+\Delta^2v+\Delta^{2\alpha}v_t=\phi_4&\in L^2(\Omega). 
\end{align*}
Using \eqref{dissipa} in the above resolvent equation, gives %{\RG [Clarify what this "Above" refers to.]}  
\begin{equation}\label{resol}
    \bigl(1-|\beta|\lambda_{\Omega}^{2(\theta-\alpha)}\bigr)\bigl[\|(-\Delta)^\alpha u_t\|^2+\|(-\Delta)^\alpha v_t\|^2\bigr] \le \Re\langle \Phi,U\rangle_{\mathcal{H}}.
\end{equation}
%{\RG [Equation~\eqref{dissipa} is an inequality. How come do we get here an identity?]\\}
 Observe that $0\in\rho(A+B)$. %{\RG [Add this notation of resolvent set to the beginning of the section, unless it is used only here, or here for the first time?]} 
 This can be verified by solving the stationary equation:  
$$
\begin{cases}
u_t=\phi_1,\\
\Delta^2u+\Delta^{2\alpha}u_t+\beta\Delta^{2\theta}v_t=\phi_2,
\end{cases}
\qquad
\begin{cases}
v_t=\phi_3,\\
\Delta^2v+\Delta^{2\alpha}v_t+\beta\Delta^{2\theta}u_t=\phi_4,
\end{cases}
$$
which implies  
\begin{equation*}
\Delta^2u=\phi_2-\Delta^{2\alpha}\phi_1-\beta\Delta^{2\theta}\phi_3,\qquad 
\Delta^2v=\phi_4-\Delta^{2\alpha}\phi_3-\beta\Delta^{2\theta}\phi_1.
\end{equation*}
This can be expressed in the variational formulation  
\[
\mathscr{B}\big((u,v),(w,z)\big)=h((w,z)), \qquad \forall\,(w,z)\in V:=H_0^2(\Omega)\times H_0^2(\Omega),
\]
%{\RG [Also in this section, we use $b$ to denote multiple things, since it was already used in~\eqref{degcase} as a constant. Please use a different notation for this sesquilinear form.]}
where the sesquilinear form $\mathscr{B}:V\times V\to\mathbb C$ is defined by  
\begin{equation*}
\mathscr{B}\big((u,v),(w,z)\big):=\langle \Delta u,\Delta w\rangle+\langle \Delta v,\Delta z\rangle,
\end{equation*}
which is coercive, since  
\[
\mathscr{B}\big((u,v),(u,v)\big)=\|\Delta u\|+\|\Delta v\|\gtrsim\|(u,v)\|_{V}^2.
\]
%{\RG [Here you are specifying the $L^2$ in the norm $\|\,\cdot\,\|_{L^2}$, before you were using only $\|\,\cdot\,\|$. Make the notation consistent. Besides, make sure to introduce the notation $\gtrsim$ at the beginning of the section or whenever it appears for the first time, if that is used seldom in the paper.]\\}
The right-hand side functional is given by  
\[
h((w,z)) := \big\langle (\phi_2-\Delta^{2\alpha}\phi_1-\beta\Delta^{2\theta}\phi_3,\, \phi_4-\Delta^{2\alpha}\phi_3-\beta\Delta^{2\theta}\phi_1),\,(w,z)\big\rangle,
\]
which is continuous on $V$. Therefore, by the Lax-Milgram Theorem, there exists a unique solution $U\in D(A)$. %{\RG [From here it looks like $D(A)$ and $D(B)$ are not related and so $U\in D(A+B)$.]}
Moreover, it is easy to see that the estimate $\|U\|_{\mathcal H}\lesssim \|\Phi\|_{\mathcal H}$ holds, which shows that $0\in\rho(A+B)$. With this in mind, we next claim that the semigroup generated by $A+B$ is strongly stable.

To this aim, we rely on the following characterization of strong stability of a semigroup~\cite{ArendtBatty1988}: the semigroup $(T(t))_{t\ge 0}$ generated by $\tilde{A}$ is strongly stable if and only if $i\R\subset \rho(\tilde{A})$. We argue by contradiction, and assume that $i\R\not\subset\rho(A+B)$.  
Let $\{\lambda_j\}_{j\in\N}\subset\Lambda :=\{h\in\R^{+}: (-ih,ih)\subset\rho(A+B)\}$ such that $\lambda_j\rightarrow \text{sup}(\Lambda)<+\infty$. That is,
$$
\lim_{j\rightarrow\infty}\|(i\lambda_jI-(A+B))^{-1}\|_{\mathcal{L}}=\infty,
$$
which means that there exists a bounded sequence $\{\Phi_j\}_j\subset\mathcal{H}$ such that
$$
\lim_{j\rightarrow\infty}\|(i\lambda_jI-(A+B))^{-1}\Phi_j\|_{\mathcal{H}}=\infty.
$$
Setting $\hat U_j:=(\hat u_j,\hat {\textbf{u}}_{j},\hat v_j,\hat {\textbf{v}}_{j})^T,$ 
$$
\hat{U}_j:=\frac{(i\lambda_jI-(A+B))^{-1}\Phi_j}{\|(i\lambda_jI-(A+B))^{-1}\Phi_j\|_{\mathcal{H}}}, \qquad
\hat\Phi_j:=\frac{\Phi_j}{\|(i\lambda_jI-(A+B))^{-1}\Phi_j\|_{\mathcal{H}}},
$$
we see that, 
%{\RG [Elaborate further why the second limit holds true. Is $\{\Phi_j\}_j\subset\mathcal{H}$ bounded?]}
$$\|\hat U_j\|_{\mathcal{H}} = 1,\qquad
i\lambda_j\hat U_j-(A+B)\hat U_j=\hat \Phi_j \qquad \text{and} \qquad \lim_{j\to\infty}\|\hat\Phi_j\|_{\mathcal{H}}=0.
$$
From \eqref{resol}, we have 
\begin{multline*}
\lim_{j\to\infty}
\bigl(1-|\beta|\lambda_{\Omega}^{2(\theta-\alpha)}\bigr)
\Bigl[\|(-\Delta)^\alpha \hat {\textbf{u}}_{j}\|^2
+\|(-\Delta)^\alpha \hat {\textbf{v}}_{j}\|^2\Bigr]
\le \lim_{j\to\infty}\Re\langle\hat \Phi_j,\hat U_j\rangle_{\mathcal{H}}\\
\le \lim_{j\to\infty}\|\hat \Phi_j\|_{\mathcal{H}}\|\hat U_j\|_{\mathcal{H}}
= 0.
\end{multline*}
%{\RG [The limit should be just equal to zero, why $\le$?]}
{which  implies $(-\Delta)^\alpha \hat {\textbf{u}}_{j},\, (-\Delta)^\alpha \hat {\textbf{v}}_{j}\to 0$ as $j\rightarrow\infty.$ Since $\|\hat{\mathbf u}_j\| \leq \lambda_\Omega^{-\alpha}\|(-\Delta)^\alpha \hat{\mathbf u}_j\|$ and $\|\hat{\mathbf v}_j\| \leq \lambda_\Omega^{-\alpha}\|(-\Delta)^\alpha \hat{\mathbf v}_j\|$, 
% \begin{align*}
%     \|\hat{\mathbf u}_j\|
% \leq \lambda_\Omega^{-\alpha}\|(-\Delta)^\alpha \hat{\mathbf u}_j\|,
% \\
% \|\hat{\mathbf v}_j\|
% \leq \lambda_\Omega^{-\alpha}\|(-\Delta)^\alpha \hat{\mathbf v}_j\|.
% \end{align*}
we deduce that $\hat{\mathbf u}_j, \hat{\mathbf v}_j \to 0$ as $j\to\infty$.} Also, since each component of $\hat{\Phi}_j$ tends to zero, it follows from the relation 
$i\lambda_j \hat{U}_j - (A+B)\hat{U}_j = \hat{\Phi}_j$ that 
$\hat{u}_j, \hat{v}_j \to 0$ as well. Consequently, $\hat{U}_j \to 0$ as $j \to \infty.$
 %{\RG [Why so? Make sure to provide justifications or references for all statements]}. 
Furthermore,  the resolvent bound
$$
\|(A+B)\hat U_j\|_{\mathcal{H}}\le\|\hat \Phi_j\|_{\mathcal{H}}+\lambda_j\|\hat U_j\|_{\mathcal{H}}\le C,
$$ implies \begin{align*}
   \hat {\textbf{u}}_{j}\quad\text{is bounded in} &\quad H^4(\Omega)\cap H^2_0(\Omega)\\
    -\Delta^2\hat u_j-\Delta^{2\alpha}\hat {\textbf{u}}_{j}-\beta\Delta^{2\theta}\hat {\textbf{v}}_{j}\quad\text{is bounded in} &\quad L^2(\Omega)\\
    \hat {\textbf{v}}_{j} \quad\text{is bounded in} &\quad H^4(\Omega)\cap H^2_0(\Omega)\\
    -\beta\Delta^{2\theta}\hat {\textbf{u}}_{j}-\Delta^2\hat v_j-\Delta^{2\alpha}\hat {\textbf{v}}_{j}\quad\text{is bounded in} &\quad L^2(\Omega),  
\end{align*}
which gives the existence of a subsequence $\hat U_{j_k}$ strongly converging to $U\in\mathcal{H}$ with $\|U\|_{\mathcal{H}}=1$. Since $A+B$ is closed, we have $i\lambda U-(A+B)U=0,$ which implies $U=0$, a contradiction. Thus, it follows that $i\mathbb{R} \subset \rho(A+B)$, which means $\left(e^{t(A+B)}\right)_{t\geq 0}$ is strongly stable. 

From \cite{ChenTriggiani1989,Chen-Triggiani1990}, we know that $A$ generates an analytic semigroup for $\alpha \in [1/2,1]$, whereas it generates a semigroup of Gevrey class $\delta>1/(2\alpha)$ for $\alpha \in (0,1/2)$. Moreover, $e^{At}$ is exponentially stable for any $\alpha\in (0,1]$ \cite{Chen-Triggiani1990}. By the Rellich-Kondrachov compact embedding theorem, the natural embedding
$$
H^{4}(\Omega)\,\cap\,H^{2}_{0}(\Omega)
   \,\to\,
   H^{4-4\theta}(\Omega)\hookrightarrow L^2(\Omega),
   \qquad 0<\theta<1/2,
   $$
   is compact.  Hence, for \(t>0\), the operator \(B\,e^{A t}\) is compact on \(\mathcal H\). As a consequence, we can state the following stability results:
\begin{itemize}
    \item For $\alpha \in \left[\frac{1}{2},1\right]$, $\theta \in \left(0,\frac{1}{2}\right)$, and $|\beta| < \lambda_{\Omega}^{2(\alpha-\theta)}$, exponential stability of the semigroup generated by $A + B$ follows from Theorem \ref{thmanalyticgib}.
    \item For $0 < \theta < \alpha < \frac{1}{2}$ and $|\beta| < \lambda_{\Omega}^{2(\alpha-\theta)}$, exponential stability of the semigroup generated by $A + B$ follows from Corollary \ref{gev-theta-sub}.
\end{itemize}

Therefore, if $|\beta| < \lambda_{\Omega}^{2(\alpha-\theta)}$, the solution of system \eqref{plateeq} decays exponentially in time, either for
$0<\theta<\alpha<\tfrac12$ or for $0<\theta<\tfrac12 \leq \alpha \leq 1$.
%{\RG [Here we shall refer to the stability of the solution to~\eqref{plateeq} (and can link that to the stability of the semigroup generated by $A+B$).]}
\subsection{Generalized coupled system of the
Kirchhoﬀ-Love Plates and Membrane-Like Electric Network}\label{couplekirchhof-elect}

We start by introducing the following coupled system formed by a Kirchhoff--Love plate equation and a membrane-like electric network~\cite{Suarez-Mendes-2025}:
\begin{equation}\label{suarezsys}
    \left\{\begin{aligned}
u_{tt} + \alpha A^2 u + \gamma A v_t &= 0, && x\in\Omega,\; t>0,\\[1ex]
v_{tt} + \beta A v - \gamma A u_t + \delta A^{\theta} v_t &= 0, && x\in\Omega,\; t>0,
\end{aligned}\right.
\end{equation} where $A: D(A)\subset L^{2}(\Omega)\to L^{2}(\Omega)$, $A=-\Delta$, and $D(A)=H^{2}(\Omega)\cap H^{1}_{0}(\Omega)$. 
In this model, only the electrical component is dissipative, with damping of fractional order $\theta\in(0,1]$. 
It is proven in~\cite{Suarez-Mendes-2025} that the associated semigroup is not analytic
for $\theta\in[0,1)$, is of Gevrey class of order $\delta>1/\theta$ for
$\theta\in(0,1)$, but becomes analytic for $\theta=1$, and is exponentially stable
for all $\theta\in[0,1]$.

 %{\RG [What's the regularity for $\theta = 1$?]}
%{\RG [Even for $\theta = 0$?]}\textcolor{magenta}{[Yes]}

We now consider the generalized system formed by a Kirchhoff-Love plate with a fractional dissipative term $\delta_1(-\Delta)^{\theta_1}\partial_t u,$ where $\theta_1\in(0,1)$ and a membrane-like electric network with fractional dissipation given by $\delta_2 (-\Delta)^{\theta_2}\partial_t v,$ where $\theta_2\in(0,1]$, as
\begin{equation}\label{kir-lov-elect-eq}
\left\{
\begin{aligned}
\partial_{t}^2u + \alpha \Delta^2 u - \gamma \Delta (\partial_t v) + \delta_1(-\Delta)^{\theta_1}\partial_t u&= 0 \,\,\,\,\text{in } \Omega \times (0, \infty), \\
\partial_{t}^2v - \beta \Delta v + \gamma \Delta (\partial_t u) + \delta_2 (-\Delta)^{\theta_2}\partial_t v &= 0 \,\,\,\,\text{in } \Omega \times (0, \infty), \\
u = \Delta u = 0,\quad v &= 0 \,\,\,\,\text{on } \Gamma \times (0, \infty),
\end{aligned}
\right.
\end{equation}
%{\RG [We shall be consistent in using $u_t$ and $u_{tt}$ in place of $\partial_t u$ and $\partial^2_{tt} u$. Or did we clarify that we used one or the other notation interchangeably?]} \textcolor{magenta}{[we already mentioned that in the notations]}
where $\alpha, \beta, , \delta_1, \delta_2 >0$ and $\gamma \geq 0$. We show that such additional Kelvin-Voigt-type damping in the plate is, in fact, relatively bounded with respect to the system \eqref{suarezsys}, and hence, by applying the results obtained in Section~\ref{sec:main}, we establish the exponential stability of this system. For physical significance of these models, we refer the interested reader to~\cite{Alessandroni-DellIsola-Porfiri-2002, Vidoli-DellIsola-2001}.

We define $U := (u, u_t, v, v_t)^\top$ and the complex Hilbert space $\mathcal{H}:= [H^2(\Omega)\cap H^1_0(\Omega)]\times L^2(\Omega)\times H^1_0(\Omega)\times L^2(\Omega)$, with norm $$\|U\|^2_{\mathcal{H}}=\alpha\|\Delta u\|^2+\|u_t\|^2+\beta\|(-\Delta)^{1/2}v\|^2+\|v_t\|^2.$$ Then, system~\eqref{kir-lov-elect-eq} can be represented in the abstract form
$$
\partial_t U= (A+B)U,
$$ where 
$$
AU :=\Bigl(u_t, -\alpha \Delta^2 u +\gamma \Delta  v_t, v_t, \beta \Delta v-\gamma\Delta u_t- \delta_2(-\Delta)^{\theta_2} v_t\Bigr),
$$ 
$$BU := \Bigl(0,- \delta_1(-\Delta)^{\theta_1}u_t,0,0\Bigr), $$ with respective domains
$$
\begin{aligned}
D(A) := \Bigl\{\, U \in \mathcal{H} \ : \ 
&\big( u_t, \ \alpha \Delta u - \gamma v_t, \ -\beta v - \delta_2 (-\Delta)^{\theta_2-1} v_t \big) 
    \in \big[ H^2(\Omega) \cap H^1_0(\Omega) \big]^3, \\
& v_t \in H^1_0(\Omega) \Bigr\},
\end{aligned}
$$
$$D(B):= \big\{U\in \mathcal{H} : u_t\in D\big((-\Delta)^{\theta_1}\big)\big\}.$$ Further, $B$ is ${\theta_1}$-subordinate to $A$. Indeed, $D(B)\supset D(A),$ and 
\begin{align*}
\|BU\|^2_{\mathcal{H}}
&=\|\delta_1(-\Delta)^{\theta_1} u_t\|^2  \\
&\lesssim  \|(-\Delta) u_t\|^{2{\theta_1}}\|u_t\|^{2(1-{\theta_1})}
%\\ &
\lesssim  \|A U\|^{2{\theta_1}}_{\mathcal{H}} \|U\|^{2(1-{\theta_1})}_{\mathcal{H}},
\end{align*}
where we have used the interpolation inequality and the definition of the norm. In \cite{Suarez-Mendes-2025}, it is proved that the semigroup generated by $A$ is exponentially stable for $\theta_2\in[0,1]$, is of Gevrey class $s > 1/\theta_2 $ for $\theta_2\in(0,1)$ and analytic for $\theta_2=1$. Therefore, applying Corollary \ref{gev-theta-sub} and Theorem \ref{thmanalyticgib}, we can say that the semigroup generated by $A+B$ is exponentially stable, provided that it is strongly stable. 

Note that $D(A)$ is dense in $\mathcal{H}$ and, for any $U\in D(A)$, we have
\begin{align*}
    \Re\langle (A+B)U,U\rangle = -\delta_1\|(-\Delta)^{\theta_1/2}u_t\|^2-\delta_2\|(-\Delta)^{\theta_2/2}v_t\|^2 <0.
\end{align*}
Thus, $A+B$ is dissipative. We now check that $0\in \rho(A+B).$ For $\Phi:=(\phi_1,\phi_2,\phi_3,\phi_4)$, consider the problem $(A+B)U=\Phi$, % i.e., the resolvent system $(i\lambda-(A+B))U=\Phi$ with $\lambda =0$, yields  
which yields
\begin{equation}
\left\{
\begin{aligned}
 u_t &= \phi_1,\\
 v_t &= \phi_3,\\
 \alpha\,\Delta^2 u &= \gamma\,\Delta \phi_3 - \phi_2 - \delta_1\,(-\Delta)^{\theta_1} \phi_1,\\
 \beta\,\Delta v &= \gamma\,\Delta \phi_1 + \delta_2\,(-\Delta)^{\theta_2} \phi_3 + \phi_4.
\end{aligned}
\right.
\label{eqs}
\end{equation}  
Define the set $Z = H^2(\Omega)\,\cap \,H^1_0(\Omega)\,\times\, H^1_0(\Omega)$ and the sesquilinear form $\mathscr{B}((u,v);$ $(w,z)) : Z\times Z\rightarrow \C$ by
$$\mathscr{B}((u,v);(w,z)) := \alpha\langle \Delta u, \Delta w\rangle + \beta\langle(-\Delta)^{1/2}v,(-\Delta)^{1/2}z \rangle ,$$  
which is coercive, since
$$\mathscr{B}((u,v);(u,v)) = \alpha\|\Delta u\|^2+\beta\|(-\Delta)^{1/2}v\|^2\gtrsim \|(u,v)\|^2_{Z}.$$  
Define the linear functional $h((w,z))$ on $Z$ as %{\RG [Specify the domain of $h$.]} 
$$h((w,z)) := \big\langle( \gamma\,\Delta \phi_3 - \phi_2 - \delta_1\,(-\Delta)^{\theta_1} \phi_1,\; \gamma\,\Delta \phi_1 + \delta_2\,(-\Delta)^{\theta_2} \phi_3 + \phi_4), (w,z)\big\rangle ,$$  
which allows us to write \eqref{eqs} in variational form:  
$$\mathscr{B}((u,v);(w,z)) = h((w,z))\quad\forall\,(w,z)\in Z.$$  
Due to the Lax-Milgram theorem, there exists a unique solution $U\in\mathcal{H}$ of $(A+B)U=\Phi$ satisfying \eqref{eqs} weakly and hence $U\in D(A+B).$ Furthermore, taking the inner product with $u$ and $v$ in the third and fourth equations of \eqref{eqs}, respectively, and applying the Cauchy-Schwarz and Young inequalities, one can easily show that $\|U\|_{\mathcal{H}}\lesssim\|\Phi\|_{\mathcal{H}}.$ Hence, $0\in \rho(A+B).$ Next, we prove strong stability of the semigroup generated by $A+B$ below.  

Return to the resolvent problem $(i\lambda-(A+B))U=\Phi,\, \lambda\in\R, \, \Phi\in\mathcal{H}$, which is given by
\begin{align*}
    i\lambda u-u_t &=\phi_1, \\
    i\lambda u_t+\alpha \Delta^2 u -\gamma \Delta  v_t+ \delta_1(-\Delta)^{\theta_1} u_t &=\phi_2,\\
    i\lambda v-v_t &= \phi_3,\\
    i\lambda v_t-\beta \Delta v+\gamma\Delta u_t+ \delta_2(-\Delta)^{\theta_2} v_t &=\phi_4, 
\end{align*}
to see that 
\begin{equation}\label{45}
    \delta_1\|(-\Delta)^{\theta_1/2}u_t\|^2+ \delta_2\|(-\Delta)^{\theta_2/2}v_t\|^2 = \Re\langle\Phi, U\rangle \leq \|\Phi\|_{\mathcal{H}}\|U\|_{\mathcal{H}}.
\end{equation}
Following a similar argument by contradiction as in the previous section, we assume $\{\lambda_j\}_{j\in\N}\subset\Lambda :=\{h\in\R^{+}: (-ih,ih)\subset\rho(A+B)\}$ such that $\lambda_j\rightarrow \text{sup}(\Lambda)<+\infty$, which implies that there exists $\{\Phi_j\}_{j\in\N}\subset\mathcal{H}$ such that
$$
\lim_{j\rightarrow\infty}\|(i\lambda_jI-(A+B))^{-1}\Phi_j\|_{\mathcal{H}}=\infty.$$
From \eqref{45} and using straightforward arguments as in Section \ref{coupledplate}, we arrive at a contradiction, and thus, $i\R\subset\rho(A+B),$
obtaining strong stability of $\left(e^{t(A+B)}\right)_{t\geq 0}$ by Arendt-Batty criterion \cite{ArendtBatty1988}. Therefore, we can summarize the findings as follows:
\begin{itemize}
    \item For $0<\theta_1<1$ and $\theta_2=1$, $\left(e^{t(A+B)}\right)_{t\geq 0}$ is exponentially stable by Theorem \ref{thmanalyticgib}.
    \item For $0<\theta_1< \theta_2$ and $0<\theta_2<1$, $\left(e^{t(A+B)}\right)_{t\geq 0}$ is exponentially stable by Corollary \ref{gev-theta-sub}.
\end{itemize}
Hence, the solution of the generalized coupled system \eqref{kir-lov-elect-eq}, of the Kirchhoff-Love plates and a membrane-like electric network, decays exponentially for $0<\theta_1<\theta_2\leq 1$.

\appendix
%%%%%%%%%%%%%%%%%%%%%%%%%%%%%%%%%%%%%%%%%%%%%%%%%%%%%%%%%%%%%%%%%%%%
%%%%%%%%%%%%%%%%%%%%%%%%%%%%%%%%%%%%%%%%%%%%%%%%%%%%%%%%%%%%%%%%%%%%

%{\RG [Double check capitalized letters in the references, for example in [28].]}

\bibliography{biblio3}
\bibliographystyle{plain}

\end{document}